\documentclass[12pt]{amsart}

\usepackage{amsmath} \usepackage{amssymb} \usepackage{amsthm}
\usepackage{amscd}
\usepackage[all]{xy}
\usepackage{enumerate}
\usepackage{mathrsfs}
\usepackage{multirow}
\usepackage{bigdelim}
\usepackage[bookmarks=false]{hyperref}

\def\tb{\textbf}

\def\G{{\mathcal{G}}}

\def\O{{\mathcal{O}}}

\def\Q{{\mathbb{Q}}}

\def\ol{\overline}

\def\rank{{\mathrm{rank}}}

\theoremstyle{plain}
\newtheorem{thm}{Theorem}[section]
\newtheorem{cor}[thm]{Corollary}
\newtheorem{prop}[thm]{Proposition}

\newtheorem{lem}[thm]{Lemma}
\theoremstyle{definition}

\theoremstyle{remark}
\newtheorem{rem}[thm]{Remark}

\newtheorem*{acknowledgement}{Acknowledgments}

\begin{document}

\title[Abundance for 3-folds with non-trivial Albanese maps]{Abundance for non-uniruled 3-folds with non-trivial Albanese maps in positive characteristics}

\address{Lei Zhang\\School of Mathematical Science\\University of Science and Technology of China\\Hefei 230026,\\P.R.China}
\email{zhlei18@ustc.edu.cn, zhanglei2011@snnu.edu.cn}
\author{Lei Zhang}
\maketitle

\begin{abstract}
In this paper, we prove abundance for non-uniruled 3-folds with non-trivial Albanese maps, over an algebraically closed field of characteristic $p > 5$. As an application we get a characterization of abelian 3-folds.
\end{abstract}

\section{Introduction}

Over an algebraically closed field of characteristic $p > 5$, existence of minimal models of 3-folds has been proved by Birkar, Hacon and Xu (\cite{Bir16} \cite{HX15}).
A natural problem is abundance: for a minimal klt pair $(X, B)$, is $K_X + B$ semi-ample?
The answer is positive when $K_X + B$ is big or $B$ is big (\cite{Bir16} \cite{CTX15} \cite{Xu15}).
This paper proves abundance for a non-uniruled 3-folds with non-trivial Albanese maps.
\begin{thm}\label{abundance}
Let $X$ be a $\Q$-factorial, projective, non-uniruled 3-fold, over an algebraically closed field of characteristic $p >5$. Let $B$ be an effective $\Q$-divisor on $X$. Assume that

(1) $(X, B)$ is a minimal klt pair; and

(2) the Albanese map $a_X: X \rightarrow A_X$ is non-trivial.\\
Then $K_X + B$ is semi-ample.
\end{thm}

\medskip

\textbf{Strategy of the proof:}
The main tools of the proof include subadditivity of Kodaira dimensions of log divisors in dimension three and explicit geometry of varieties with Kodaira dimension zero. The approach is explained below.

By the method of \cite{Wal15}, we can show abundance when $\kappa(K_X + B) \geq 1$ (Theorem \ref{s-amp-k=2}). So we only need to show that either $K_X+B\sim_{\Q} 0$ or $\kappa(X, K_X+B) \geq 1$.

If the Albanese map $a_X: X \to A_X$ is separable, then the Stein factorization of $a_X$ induces a separable fibration $f: X \to Y$. By a sequence of Frobenius base changes and a smooth resolution, we get a fibration $f': X' \to Y'$ with smooth geometric generic fiber $X'_{\ol\eta}$. We can show $\kappa(X) \geq \kappa(X') \geq 0$ by subadditivity of Kodaira dimensions (Theorem \ref{Iit-conj} and Corollary \ref{Iit-cor}). Since $B$ is effective, we only need to work on varieties with $\kappa(X) = 0$ (thus $\kappa(X') = \kappa(Y) = \kappa(X'_{\ol\eta}) = 0$). By Theorem \ref{compds} and working on a minimal model of $X'$, we reduce to computing $\kappa(W, D)$ where $W$ is a minimal model of $X'$ with $K_W \sim_{\Q} 0$ and $D$ is an effective and nef divisor on $W$. We aim to show that either $\kappa(W, D) >0$ or $D=0$, which implies our theorem.
In case $\dim Y = 1$, we will use subadditivity of Kodaira dimensions of log divisors (Corollary \ref{div-k=0}).
In case $\dim Y = 2$, the condition $\kappa(X')=0$ implies that $Y$ is an abelian surface and $\mathrm{var}(f') = 0$ by Theorem \ref{Iit-conj} (iv), we can do a base change $Z \to A$ such that, the variety $X \times_A Z$ is birational to $Z\times C$, then we treat the case $\kappa(X_{\ol\eta}, (K_X+B)_{\ol\eta}) = 0$ by using canonical bundle formula, and treat the case $\kappa(X_{\ol\eta}, (K_X+B)_{\ol\eta}) = 1$ by pulling back $D$ on $Z\times C$. In case $\dim Y = 3$, if $A_X$ is simple, then we can prove that $X$ is birational to an abelian variety (Theorem \ref{ch-of-ab-kx}); if $A_X$ is not simple, then $X$ has a natural fibration to an elliptic curve or an abelian surface, we reduce to one of the previous two cases.

If the Albanese map $a_X: X \to A_X$ is inseparable, then we have a foliation $\mathcal{F} = \mathcal{L}^{\bot} \subset T_X$ where $\mathcal{L}$ is the saturation of the image of the natural homomorphism $a_X^*\Omega_{A_X}^1 \rightarrow \Omega_{X}^1$. By replacing $X$ with $X/\mathcal{F}$ repeatedly, we can finally obtain a variety whose Albanese map is separable, then show that $\kappa(X) \geq 1$ by induction (Theorem \ref{reduction}).

\medskip

As an application, combining Theorem \ref{abundance} and Theorem \ref{ch-of-ab-kx} gives the following result.
\begin{cor}
Let $X$ be a smooth projective 3-fold of maximal Albanese dimension, over an algebraically closed field of characteristic $p > 5$. If $\kappa(X) =0$, then $X$ is birationally equivalent to an abelian 3-fold.
\end{cor}

\medskip

This paper is organized as follows. In Sec. \ref{pre}, we collect results on minimal models and subadditivity of Kodaira dimensions. In Sec. \ref{sec-s-amp-k=2}, we show abundance under the assumption that $\kappa(X, K_X + B) \geq 1$. In Sec. \ref{sec-ch-of-ab}, we treat the case that the Albanese map is inseparable and give some criteria for a variety being birational to an abelian variety. In Sec. \ref{sec-pf-abd}, we prove Theorem \ref{abundance}.

Part of the results in this paper are also proved by Das and Waldron independently around the same time by different method in \cite{DW16}.

\medskip

\textbf{Notation and Conventions:}
Let $X$ be a projective variety over a field $K$ and $D$ a $\mathbb{Q}$-Cartier divisor on $X$. The \emph{$D$-dimension} $\kappa(X,D)$ is defined as
\[\kappa(X,D) =\left\{
\begin{array}{llr}
-\infty, \text{ ~~if for every integer } m >0, |mD| = \emptyset;\\
\max \{\dim_K \Phi_{|mD|}(X)| m \in \mathbb{Z}~\text{and}~m>0 \}, \text{ otherwise.}
\end{array}\right.
\]
If $X$ has a regular projective birational model $\tilde{X}$, the Kodaira dimension $\kappa(X)$ of $X$ is defined as $\kappa(\tilde{X}, K_{\tilde{X}})$
where $K_{\tilde{X}}$ denotes the canonical divisor.

Throughout this paper, we work over an algebraically closed field $k$ with $\mathrm{char}~k = p > 0$.  A {\it variety}
means an integral separated scheme of finite type over $k$.

For the notions in minimal model theory such as lc, klt pairs, flip and divisorial contraction and so on, please refer to \cite{Bir16}.
By \cite{CP08} and \cite{CP09}, we can always take a log smooth resolution of a pair $(X, \Delta)$ in dimension three.

A {\it fibration} means a projective morphism $f:X\to Y$
between varieties such that the natural morphism $\O_Y \to f_*\O_X$ is an isomorphism. An \emph{elliptic fibration} means a fibration whose geometric generic fiber is a smooth elliptic curve.
In this paper, since the notation $f:X\to Y$ appears frequently (often as a fibration), we will use $\eta$ (resp. $\ol\eta$) specially to denote the (geometric) generic point of $Y$, and use $X_\eta$ (resp. $X_{\ol\eta}$) to denote the (geometric) generic fiber of $f$.

For a variety $X$, we use $F_X^e: X^{e} \to X$ or $F_X^e: X \to X^{(e)}$ for the $e$-th absolute Frobenius iteration.

For a normal projective variety $X$, $\mathrm{Pic}^0(X)_{\mathrm{red}}$ is an abelian variety (see e.g.\cite[Sec. 9.5]{FGA}). Let $A_X$ denote the abelian variety dual to $\mathrm{Pic}^0(X)_{\mathrm{red}}$. If $\dim \mathrm{Pic}^0(X)>0$, then there is a natural nontrivial morphism $a_X: X \rightarrow A_X$, namely, the Albanese morphism (\cite[Sec. 5]{Ba01}).

Let $\varphi:X \to T$ be a morphism of schemes and let $T'$ be a $T$-scheme.
Then we denote by $X_{T'}$ the fiber product
$X\times_{T}T'$. For a Cartier
or $\Q$-Cartier divisor $D$ on $X$ (resp. an $\O_X$-module $\G$), the pullback of $D$ (resp. $\G$) to $X_{T'}$
is denoted by $D_{T'}$ or $D|_{X_{T'}}$ (resp. $\G_{T'}$ or $\G|_{X_{T'}}$) if it is well-defined.

We use $\sim$ (resp. $\sim_{\mathbb{Q}}$) for linear (resp. $\mathbb{Q}$-linear) equivalence between Cartier (resp. $\mathbb{Q}$-Cartier) divisors and line bundles. On a normal variety $X$, for two $\Q$-divisors $D_1, D_2$, if the restrictions on the smooth locus of $X$, $D_1|_{X^{sm}} \sim_{\Q} D_2|_{X^{sm}}$, we also denote $D_1 \sim_{\Q} D_2$.

\begin{acknowledgement}
The idea was formulated during the period of the author visiting BICMR, and the paper was prepared when visiting KIAS. The author thanks the two institutes for their hospitalities and supports. He expresses his appreciation to Prof. Chenyang Xu for his encouragement and to Dr. Sho Ejiri, Yi Gu and Joe Waldron for many useful discussions. He also thanks Hiromu Tanaka for his useful comments. After this paper was finished, Joe Waldron and Omprokash Das informed me that they got similar results. The author thanks them for sharing their preprint. Finally the author is grateful to an anonymous referee for pointing me out a mistake and other valuable suggestions. The author was supported by the NSFC (No. 11771260 and 11401358).
\end{acknowledgement}

\section{Preliminaries}\label{pre}

\subsection{Separability}
Recall some results of separability.
\begin{prop}\label{prop-of-sep}
Let $f: X \to Y$ be a surjective morphism between normal quasi-projective varieties.

(1) If $f$ is a fibration, then $f$ is separable if and only if the geometric generic fiber $X_{\bar{\eta}}$ is integral, and if and only if $X_{\bar{\eta}}$ is reduced.

(2) If $f$ is a fibration to a curve, then $f$ is separable.

(3) The morphism $f$ is separable if and only if $\mathrm{rank}~\Omega_{X/Y} = \dim X -\dim Y$.

(4) Denote by $F$ the generic fiber of $f$. If $F \otimes_{K(Y)}K(Y)^{\frac{1}{p^{\infty}}}$ is regular then $F$ is smooth over $K(Y)$.
\end{prop}
\begin{proof}
For (1), since $f$ is a fibration we have $H^0(X_\eta, \O_{X_\eta}) \cong K(Y)$, and since $X$ is normal, $K(Y)$ is integrally closed in $K(X)$. Therefore $X_{\bar{\eta}}$ is irreducible by \cite[Chap.3 Cor. 2.14 (d)]{Liu10}. Then assertion (1) follows from \cite[Chap.3 Prop. 2.15]{Liu10}

For (2) refer to \cite[Lemma 7.2]{Ba01}, for (3) refer to \cite[Chap. II Proposition 8.6A]{Har77}, and for (4) refer to \cite[Chap. 4 Corollary 3.33]{Liu10} since $K(Y)^{\frac{1}{p^{\infty}}}$ is perfect.
\end{proof}
\begin{rem}\label{bc-smthing}
Assume $\dim X = 3$ and $f$ is separable. For an integer $e >0$, consider the $e$-th absolute Frobenius base change $F_Y^e: Y'= Y^e \to Y$, and take a smooth resolution $\sigma: X' \to \bar{X}'= X\times_Y Y'$. Applying Proposition \ref{prop-of-sep} (4), we see that if $e$ is sufficiently large then the geometric generic fiber of $f': X'\to Y'$ is smooth (see \cite[Proof of Corollary 1.3]{BCZ15}).
\end{rem}

\subsection{Covering Theorem}
The result below is \cite[Theorem 10.5]{Iit82} when $X$ and $Y$ are both smooth, and the proof therein also applies when the varieties are normal.
\begin{thm}\textup{(\cite[Theorem 10.5]{Iit82})}\label{ct}
Let $f\colon X \rightarrow Y$ be a proper surjective morphism between complete normal varieties.
If $D$ is a Cartier divisor on $Y$ and $E$ an effective $f$-exceptional divisor on $X$, then
$$\kappa(X, f^*D + E) = \kappa(Y, D).$$
\end{thm}

\subsection{The behavior of relative canonical divisors under base changes}
\begin{prop}\label{compds}
Let $f: X \rightarrow Y$ be a separable fibration between two normal varieties. Let $\Delta$ be an effective $\Q$-Weil divisor on $X$ such that $K_{X/Y} + \Delta$ is $\Q$-Cartier.
Let $\pi: Y' \rightarrow Y$ be a smooth modification, $\bar{X}'$ the main component of $X\times_Y Y'$ and $\sigma: X' \rightarrow \bar{X}'$ a birational projective morphism with $X'$ normal, which fit into the following commutative diagram
$$
\xymatrix@C=2cm{
&X' \ar@/^2pc/[rr]|{\sigma'}\ar[r]^>>>>>>>>>{\sigma} \ar[rd]^{f'} &\bar{X}' \subset X \times_Y Y'
\ar[r]^<<<<<<<<<{\pi'}\ar[d]^{\bar{f}'}    &X\ar[d]^f \\
& &Y'\ar[r]^{\pi}   &Y
}
$$
where $\pi'$ and $\bar{f}'$ denote the natural projections, $f'=\bar{f}'\circ \sigma$ and $\sigma' = \pi'\circ \sigma$.

Assume either that $f$ is flat or that $Y$ is smooth.
Then there exist an effective $\sigma'$-exceptional $\Q$-divisor $E'$ and an effective $\Q$-divisor $\Delta'$ on $X'$ such that
$$K_{X'/Y'} + \Delta' \sim_{\Q} \sigma'^*(K_{X/Y} + \Delta) + E'.$$
\end{prop}
\begin{proof}
If $f$ is flat, then the assertion is \cite[Proposition 2.1]{Zha16}. If $Y$ is smooth, by working on the smooth locus of $X$ and $X'$, we can prove the assertion by similar arguments of \cite[Theorem 2.4]{CZ15}
\footnote{The proof of \cite[Theorem 2.4]{CZ15} contains a mistake: in that long equation to explain the homomorphism $\beta$, the $4^{\mathrm{th}}$ ``$\cong$'' holds when $L\pi_2^*i_*\O_X$ is perfect, hence the proof is correct if $Z$ (hence $P$) is smooth, otherwise it is wrong in general. This mistake does not affect the main results of \cite{CZ15}.}.
\end{proof}

\subsection{Minimal model theory of 3-folds}
We collect some results on minimal model theory of 3-folds in the following theorem, which will be used in the sequel.
\begin{thm}\label{rel-mmp}
Assume $\mathrm{char}~k =p >5$. Let $(X,B)$ be a $\mathbb{Q}$-factorial projective pair of dimension three and $f: X\rightarrow  Y$ a projective surjective morphism.

(1) If either $(X, B)$ is klt and $K_X+B$ is pseudo-effective over $Y$, or $(X,\Delta)$ is lc and $K_X+\Delta$ has a weak Zariski decomposition over $Y$, then $(X,B)$ has a log minimal model over $Y$.

(2) If $(X, B)$ is klt and $K_X+B$ is not pseudo-effective over $Y$, then $(X,B)$ has a Mori fibre space over $Y$.

(3) Assume that $(X, B)$ is klt and $K_X+B$ is nef over $Y$.
\begin{itemize}
\item[(3.1)]
If $K_X+B$ or $B$ is big over $Y$, then $K_X+B$ is semi-ample over $Y$.
\item[(3.2)]
If $\dim Y \geq 1$, $X_{\eta}$ is integral and $\kappa(X_{\eta}, (K_X+B)_{\eta}) \geq 0$, then $(K_X+B)_{\eta}$ is semi-ample on $X_{\eta}$.
\item[(3.3)]
If $Y$ is a smooth curve, $X_{\eta}$ is integral and $\kappa(X_{\eta}, (K_X+B)_{\eta}) = 0$ or $2$, then $K_X+B$ is semi-ample over $Y$.
\item[(3.4)]
If $Y$ contains no rational curves, then $K_X+B$ is nef.
\end{itemize}

(4) If $Y$ is a non-uniruled surface and $K_X+B$ is pseudo-effective over $Y$, then $K_X + B$ is pseudo-effective, and there exists a map $\sigma: X \dashrightarrow \bar{X}$ to a minimal model $\bar{X}$ of $X$ such that, the restriction $\sigma|_{X_{\eta}}$ is an isomorphism to its image.
\end{thm}
\begin{proof}
For (1) please refer to \cite[Theorem 1.2 and Proposition 7.3]{Bir16}.

For (2), refer to \cite{BW14}.

For (3.1), please refer to \cite{Bir16}, \cite{Xu15} and \cite{BW14}.

For (3.2) and (3.3) please refer to \cite[Theorem 1.5 and 1.6 and the remark below 1.6]{BCZ15}. And (3.2) also can be obtained from \cite{Ta15}.

Assertion (3.4) follows from the cone theorem \cite[Theorem 1.1]{BW14}. Indeed, otherwise we can find an extremal ray $R$ generated by a rational curve $\Gamma$,  so $\Gamma$ is contained in a fiber of $f$ since $Y$ contains no rational curves, this contradicts that $K_X + B$ is $f$-nef.

For (4), $K_X + B$ is obviously pseudo-effective because otherwise, $X$ will be ruled by horizontal (w.r.t. $f$) rational curves by (2), which contradicts that $Y$ is non-uniruled. The exceptional locus of a flip contraction is of dimension one, so it does not intersect $X_{\eta}$, neither does that of an extremal divisorial contraction because it is uniruled (see the proof of \cite[Lemma 3.2]{BW14}). Running an LMMP for $K_X +B$, by induction we get a needed map $\sigma: X \dashrightarrow \bar{X}$.
\end{proof}

\subsection{Subadditivity of Kodaira dimensions}

Subadditivity of Kodaira dimensions of log divisors on 3-folds plays a key role in our proof. We collect some results which will be used in the sequel. For more general results on this topic please refer to \cite{Pa14}, \cite{Pa13}, \cite{Ej15}, \cite{Zha16} and \cite{Zha16b}.

\begin{thm}\label{Iit-conj}
Assume $\mathrm{char}~k =p >5$. Let $f: X \to Y$ be a separable fibration from a $\Q$-factorial projective 3-fold to a smooth projective variety of dimension $1$ or $2$. Let $B$ be an effective $\Q$-divisor on $X$ such that $(X,B)$ is klt.
If one of the following holds

(i) $\dim (Y) = 1$, $(X_{\ol\eta}, B_{\ol\eta})$ is sharply $F$-pure, Cartier index of $K_{X_{\ol\eta}} + B_{\ol\eta}$ is not divisible by $p$, and $K_X + B$ is $f$-$\Q$-trivial;

(ii) $\dim (Y) = 1$, $\kappa(X_{\ol\eta}, K_{X_{\ol\eta}} + B_{\ol\eta}) = 1$, and the Iitaka fibration $I_{\ol\eta}: X_{\ol\eta} \to C_{\ol\eta}$ of $K_{X_{\ol\eta}} + B_{\ol\eta}$ is an elliptic fibration to a normal curve $C_{\ol\eta}$;

(iii) $\dim (Y) = 1$, $(X_{\ol\eta}, B_{\ol\eta})$ is sharply $F$-pure, Cartier index of $K_{X_{\ol\eta}} + B_{\ol\eta}$ is not divisible by $p$ and $K_{X_{\ol\eta}} + B_{\ol\eta}$ is ample;

(iv) If $\dim Y = 2$, $\kappa(Y) \geq 0$ and $X_{\bar{\eta}}$ is smooth, then
$$\kappa(X) \geq \kappa(X_{\bar{\eta}}) + \mathrm{max}\{\mathrm{Var}(f), \kappa(Y)\}.$$
\end{thm}
\begin{proof}
In the following we assume $\kappa(Y) \geq 0$.

In case (i), by Theorem \ref{rel-mmp} (3.3) $K_{X/Y} + B \sim_{\Q} f^*A$, and $A$ is nef by \cite[Theorem 1.5]{Pa14}. If $g(Y) >1$ then the assertion follows by $\kappa(X, K_X +B) = \kappa(Y, K_Y +A) =1$; and if $g(Y) = 1$ then the assertion follows by \cite[Theorem 3.2]{EZ16}.

In case (ii), by running a relative LMMP over $Y$, we may assume $K_X + B$ is $f$-nef. So $K_X + B$ is nef, and $(K_X + B)|_{X_{\ol\eta}}$ is semi-ample (Theorem \ref{rel-mmp} (3.2, 3.4)) and induces an elliptic fibration $I_{\ol\eta}: X_{\ol\eta} \to C_{\ol\eta}$ to a normal curve $C_{\ol\eta}$. Applying the proof of \cite[Theorem 2.8]{EZ16} we can show that $f_*\O_X(m(K_{X/Y} +B))$ contains a nef sub-bundle of rank $\geq cm$ for some $c>0$ and any sufficiently divisible $m>0$. Then the assertion follows from the arguments of \cite[Sec. 4, Step 1-4]{EZ16} by replacing $K_X$ with $K_X+B$\footnote{Note that in Step 3, with $\tilde{B}$ defined in the same way, we have $K_{\tilde{X}} + \tilde{B} = \mu^*(K_X+B) + \tilde{B}'$ for some effective divisor $\tilde{B}'$, so all the arguments apply.}.

In case (iii), combining results of \cite[Corollary 2.23]{Pa14}, we see that all the conditions of \cite[Theorem 1.4]{Ej15} are satisfied, hence the assertion follows.

In case (iv), the assertion is \cite[Remark 3.3]{CZ15}.
\end{proof}

\begin{rem}
For the condition (ii) of the theorem above, the reason why we assume $K_{X_{\ol\eta}} + B_{\ol\eta}$ induces an elliptic fibration lies in that, to show $f_*\O_X(m(K_{X/Y} +B))$ contains a nef sub-bundle, the proof of \cite[Theorem 2.8]{EZ16} uses ``canonical bundle formula'': for a fibration $h: X \to Z$, if $K_{X/Z}$ is $\Q$-trivial over $Z$, then $K_{X/Z} \sim_{\Q} h^*\Delta$ for some effective $\Q$-divisor $\Delta$ on $Z$, which is true for elliptic fibrations by the following theorem.
\end{rem}

\begin{thm}[\textup{\cite[Claim 3.2]{CZ15}}]\label{rel-can-ellfib}
Let $h:X\to Z$ be an elliptic fibration between smooth projective varieties. Then $\kappa(X,K_{X/Z})\ge 0$.
\end{thm}

\medskip

\begin{cor}\label{Iit-cor}
Assume $\mathrm{char}~k =p >5$. Let $f: X \to Y$ be a separable fibration from a smooth projective 3-fold to a smooth projective variety of dimension $1$ or $2$.
Denote by $\tilde{X}_{\bar{\eta}}$ a smooth projective birational model of $X_{\bar{\eta}}$.
Then
$$\kappa(X) \geq \kappa(\tilde{X}_{\bar{\eta}}) + \kappa(Y).$$

\medskip

In particular, if moreover both $X$ and $Y$ are non-uniruled then $\kappa(X) \geq 0$.
\end{cor}
\begin{proof}
We can assume $\kappa(\tilde{X}_{\bar{\eta}}) \geq 0$ and $\kappa(Y)\geq 0$. By Remark \ref{bc-smthing}, we can take some $e$-th absolute Frobenius iteration $F_Y^e: Y'= Y^e \to Y$ such that, for a smooth resolution $\sigma: X' \to \bar{X}'= X\times_Y Y'$, the geometric generic fiber of $f': X'\to Y'$ is smooth.
By Proposition \ref{compds}, there exists an effective $\sigma'$-exceptional divisor $E'$ on $X'$, where $\sigma':X' \to X$ denotes the natural morphism, such that
$$K_{X'/Y'} \leq \sigma'^*K_{X/Y} + E.$$
It follows that
\begin{equation}
\begin{split}
\kappa(X', K_{X'}) & = \kappa(X', K_{X'/Y'} + f'^*K_{Y'}) \\
                   & \leq \kappa(X', \sigma'^*K_{X/Y} + E + f'^*K_{Y'})\\
                   & = \kappa(X', \sigma'^*K_{X} + E + (1 - p^e) f'^*K_{Y'}) \hspace{1cm} \text{by $\sigma'^*f^*K_Y \sim f'^*p^eK_{Y'}$} \\
                   & \leq \kappa(X', \sigma'^*K_{X} + E) =  \kappa(X, K_X) \hspace{1cm} \text{by Theorem \ref{ct}}.
\end{split}
\end{equation}
So to show the inequality of the theorem, it suffices to show that $\kappa(X') \geq \kappa(X'_{\bar{\eta}}) + \kappa(Y')$. If $\dim Y = 2$, then we are done by Theorem \ref{Iit-conj} (iv). Let's consider the case $\dim Y = 1$.
We assume $\kappa(\tilde{X}_{\bar{\eta}}) = \kappa(X'_{\bar{\eta}}) \geq 0$, then $K_{X'}$ is pseudo-effective. By running a relative MMP on $X'$ over $Y'$, we can assume $K_{X'}$ is nef by Theorem \ref{rel-mmp} (3.4),
and if $\kappa(X'_{\bar{\eta}}) = 2$ we assume $K_{X'_{\bar{\eta}}}$ is ample by considering the relative canonical model. Then $X'_{\bar{\eta}}$ has at most canonical singularities, which is strongly $F$-regular by \cite{Har98} and hence is sharply $F$-pure.
Note that since $p>5$, in case $\kappa(X'_{\bar{\eta}}) = 1$ the Iitaka fibration $I_{\ol\eta}: X'_{\ol\eta} \to C_{\ol\eta}$ is an elliptic fibration (\cite[Theorem 7.18]{Ba01}).
By Theorem \ref{rel-mmp} (3.2, 3.3) we can apply Theorem \ref{Iit-conj} to the fibration $f': X' \to Y'$ to show the subadditivity.

For the remaining assertion, assume that both $X$ and $Y$ are non-uniruled. Then $\tilde{X}_{\bar{\eta}}$ is non-uniruled. By \cite[Theorem 13.2]{Ba01} $\kappa(\tilde{X}_{\bar{\eta}})$ and $\kappa(Y)$ are non-negative. Applying the subadditivity above gives that $\kappa(X) \geq 0$.
\end{proof}

\begin{cor}\label{div-k=0}
Assume $\mathrm{char}~k =p >5$. Let $X$ be a normal $\Q$-factorial klt projective 3-fold with $K_X \sim_{\Q} 0$, and let $D$ be an effective nef $\Q$-divisor on $X$.
Assume that $X$ has a morphism $f: X\to Y$ to an elliptic curve and that $X_{\ol\eta}$ has at most canonical singularities. Then either $D=0$ or $\kappa(X, D) \geq 1$.
\end{cor}
\begin{proof}
We can replace $D$ by $tD$ for a sufficiently small rational number $t>0$ and assume that both $(X, D)$ and $(X_{\bar{\eta}}, D_{\bar{\eta}})$ are klt, then replace $D$ by $\frac{p^n}{p^n+1}D$ for sufficiently large $n$ to assume that $D$ has Weil index not divisible by $p$.

If $\kappa(X_{\eta}, D_{\eta}) = 0$, then by Theorem \ref{rel-mmp} (3.3) $D \sim_{\Q} K_{X} + D \sim_{\Q} f^*A$, and $A$ can be assumed effective since $D$ is effective. We conclude that either $D =0$ or that $\kappa(X, D) = 1$.

If $\kappa(X_{\eta}, D_{\eta}) = 1$, then $D_{\ol\eta} \sim_{\Q} K_{X_{\ol\eta}} + D_{\ol\eta}$ is semi-ample by Theorem \ref{rel-mmp} (3.2). Denote the associated map to $D_{\ol\eta}$ by $I_{\ol\eta}: X_{\ol\eta} \to C_{\ol\eta}$, and by $G$ a general fiber of $I_{\ol\eta}$ which has arithmetic genus $p_a(G) = 1$ by adjunction formula. Since $\mathrm{char}~k>5$ and $X_{\ol\eta}$ is an algebraic surface with at most canonical singularities (hence normal), we have that $C_{\ol\eta}$ is normal, and by \cite[Theorem 7.18]{Ba01} $I_{\ol\eta}: X_{\ol\eta} \to C_{\ol\eta}$ is an elliptic fibration. So applying Theorem \ref{Iit-conj} (ii), we conclude that
$$\kappa(X, D) = \kappa(X, K_X +D) \geq 1.$$

If $\kappa(X_{\eta}, D_{\eta}) = 2$, we consider the relative log-canonical model $(X', D')$ (Theorem \ref{rel-mmp} (3.1)). We can check that $(X'_{\bar{\eta}}, D'_{\bar{\eta}})$ is klt, in particular $X'_{\bar{\eta}}$ has klt singularities, hence is strongly $F$-regular (\cite{Har98}). Replacing $D'$ with the multiplication by a small rational number, we can assume $(X'_{\bar{\eta}}, D'_{\bar{\eta}})$ is strongly $F$-regular (\cite[Lemma 2.8]{HX15}). Finally applying Theorem \ref{Iit-conj} (iii), we  conclude that
$$\kappa(X, D) = \kappa(X, K_X +D) = \kappa(X', K_{X'} +D') \geq 2.$$

In conclusion, the proof is completed.
\end{proof}

\subsection{Foliations and purely inseparable morphisms}
Let $X$ be a smooth variety. Recall that a (1-)\emph{foliation} is a saturated subsheaf $\mathcal{F} \subset T_X$ which is involutive (i.e., $[\mathcal{F}, \mathcal{F}] \subset \mathcal{F}$) and $p$-closed (i.e., $\xi^p \in \mathcal{F}, \forall \xi \in \mathcal{F}$). A foliation $\mathcal{F}$ is called \emph{smooth} if it is locally free.
Denote
$$Ann(\mathcal{F}) = \{a \in \O_X| \xi(a) = 0, \forall \xi \in \mathcal{F}\}.$$

\begin{prop}\label{flt}
Let $X$ be a smooth variety and $\mathcal{F}$ a foliation on $X$.

(1) We get a normal variety $Y = X/\mathcal{F}= \mathrm{Spec} Ann(\mathcal{F})$, and there exist natural morphisms $\pi:X \to Y$ and $\pi': Y \to X^{(1)}$ fitting into the following commutative diagram
$$
\xymatrix{
&X\ar[d]^{\pi}\ar[dr]^{F_X}  &\\
&Y\ar[r]^>>>>>{\pi'}  &X^{(1)}.
}
$$
Moreover $\deg \pi = p^r$ where $r = \rank~ \mathcal{F}$.

(2) There is a one-to-one correspondence between foliations and normal varieties between $X$ and $X^{(1)}$, by the correspondence $\mathcal{F} \mapsto X/\mathcal{F}$ and the inverse correspondence $Y \mapsto Ann(\O_Y)$.

(3) The variety $Y$ is regular if and only if $\mathcal{F}$ is smooth.

(4) If $Y_0$ denotes the regular locus of $Y$ and $X_0 = \pi^{-1}Y_0$, then
$$K_{X_0} \sim \pi^*K_{Y_0} + (p-1)\det \mathcal{F}|_{X_0}.$$
\end{prop}
\begin{proof}
Refer to \cite[p.56-58]{MP97} or \cite{Ek87}.
\end{proof}

\section{Abundance for 3-folds with $\kappa(X) \geq 1$}\label{sec-s-amp-k=2}

The following result in case $\kappa(X, K_X + B) = 2$ has been proved by Waldron in \cite{Wal15}, where he obtains some results in arbitrary dimension.
In characteristic zero, similar results have been proved by Kawamata by using Kollar's vanishing (\cite[Theorem 6.1]{Ka85I}). Here for readers' conveniences, we borrow Waldron's idea and give a quick proof.
\begin{thm}\label{s-amp-k=2}
Assume $\mathrm{char}~k =p >5$. Let $(X, B)$ be a $\Q$-factorial klt projective 3-fold. Assume that $K_X + B$ is nef. If $\kappa(X, K_X + B) \geq 1$, then $K_X + B$ is semi-ample.
\end{thm}

\medskip

As an important preparation, we prove the following lemma, which is included in \cite{Wal15}. For readers' conveniences, the proof is sketched.
\begin{lem}\label{flatten}
Assume $\mathrm{char}~k =p >5$ and $k$ is uncountable. Let $(X, B)$ be a $\Q$-factorial klt projective 3-fold. Assume that $K_X + B$ is nef with $\kappa(K_X + B) = \nu(K_X + B) = 2$. Then $K_X +B$ is endowed with a morphism $h: X \to Z$ to an algebraic space $Z$ of dimension two, and there exists another $\Q$-factorial minimal model $(X^+, B^+)$ of $(X, B)$ over $Z$ such that, the natural morphism $h^+: X^+ \to Z$  is equi-dimensional.
\end{lem}

\medskip

\emph{Sketch of the proof.}
Since $\kappa(K_X + B) = \nu(K_X + B) = 2$, by \cite[Lemma 7.2]{BW14} the divisor $K_X + B$ is endowed with a morphism $h: X \to Z$ to an algebraic space $Z$ of dimension two. Applying \cite[Proposition 2.1]{Ka85I} and flattening trick (see also \cite[Lemma 5.6]{BW14}), we get the following commutative diagram
$$ \xymatrix{ &X_1\ar[r]^{\phi}\ar[d]_{h_1} &X\ar[d]^h\\
&Z_1 \ar[r]^{\psi} & Z } $$
where $Z_1$ is a smooth projective surface, $X_1$ is a normal projective $3$-fold, $\phi, \psi$ are birational morphisms and $h_1: X_1 \to Z_1$ is an equi-dimensional fibration such that,
there exists a nef and big $\Q$-divisor $D_1$ on $Z_1$ satisfying that $\phi^*(K_X+B) \sim_{\Q} h_1^*D_1$.

Denote by $E(D_1)$ the exceptional locus of $D_1$, which equals to the union of $D_1$-numerically trivial curves. Let $T = \psi (E(D_1))$, which consists of finitely many closed points on $Z$. Then
$$U = Z \setminus T \cong Z_1 \setminus E(D_1)$$
is an algebraic variety, $X_U \cong (X_1)_U$ is equi-dimensional over $U$ and $K_{X_U} + B|_{X_U}$ is $\Q$-linearly trivial over $U$.

We contract all divisors on $X$ over $T$ by running a minimal model program over $Z$. The process is explained below, please refer to \cite{Wal15} for details.

Step 1: Let $F$ be a prime divisor on $X$ such that $h(F) \in T$. Then $F$ is not nef.
Fix a rational number $\epsilon >0$ such that $(X, B+ \epsilon F)$ is klt. Then we can run an LMMP for $K_X + B+ \epsilon F$ over $Z$. For the first step, the $K_X + B + \epsilon F$-extremal ray is $K_X + B$-numerically trivial. After a divisorial contraction or a $K_X + B + \epsilon F$-flip, we get $\mu_1: (X, B+ \epsilon F) \dashrightarrow (X_1^+, B_1^+ + \epsilon F_1^+)$. If $\mu_1$ is a divisorial contraction then $F_1^+ = 0$ and the LMMP terminates; otherwise, we find that $F_1^+$ is not nef and the $K_{X_1^+} + B_1^+ + \epsilon F_1^+$-extremal ray is $K_{X_1^+} +B_1^+$-numerically trivial. By induction, the log minimal model program ends up with a pair $(X_r^+, B_r^+)$ such that, $K_r^+ + B_r^+$ is nef, $F$ is contracted by the birational map $X \dashrightarrow X_r^+$ and $X_r^+ \dashrightarrow X$ has no exceptional divisors.

Step 2: If $X_r^+ \to Z$ is not equi-dimensional, then proceeding with the process in Step 1 on $(X_r^+, B_r^+)$, after finitely many steps we can get a pair $(X^+, B^+)$ satisfying all the conditions in the lemma.

\medskip

\begin{proof}[of Theorem \ref{s-amp-k=2}.]
We can pass to an uncountable field.
If the numerical dimension $\nu(K_X + B) = 3$, then the assertion follows from Theorem \ref{rel-mmp} (3.1).
So from now on, we assume $\nu(K_X + B) = 1$ or $2$, which means that $\kappa(X, K_X + B) =1$ or $2$ by the assumption. There exists a log smooth resolution $\mu: X'\to X$ of $(X, B)$ such that, the Iitaka fibration $g': X' \rightarrow Z'$ is a morphism. Let $B'= \mu_*^{-1}B + (1-\epsilon) E$, where $E$ is the sum of all $\mu$-exceptional divisors, and $0< \epsilon \ll 1$ is sufficiently small such that $(X,B)$ is a minimal model of $(X', B')$. Let $g'': (X'', B'') \to Z'$ be a minimal model of $(X', B')$ over $Z'$. Denote by $G''$ the generic fiber of $g''$. Then $\kappa(G'', K_{G''} + B''|_{G''}) = 0$, thus $K_{G''} + B''|_{G''} \sim_{\Q} 0$ by Theorem \ref{rel-mmp} (3.2).

If $\kappa(X, K_X + B) = 1$, then $K_{X''} + B''$ is $\Q$-trivial over $Z'$ by Theorem \ref{rel-mmp} (3.3). We can assume that $K_{X''} + B'' \sim_{\Q} g''^{*} H$ where $H$ is an ample $\Q$-divisor on $Z'$, thus $K_{X''} + B''$ is semi-ample. As $(X'', B'')$ is another minimal model of $(X', B')$, by a standard argument using the negativity lemma, the pullback of $K_{X} + B$ and $K_{X''} + B''$ coincide on any common resolution of $X$ and $X''$ (cf. \cite[Remark 2.7]{Bi12}). Therefore, $K_X+B$ is semi-ample.

We are left to consider the case $\kappa(X, K_X + B) = 2$. Then $\nu(K_X + B) = 2$. Let the notation be as in the proof of Lemma \ref{flatten}.
Recall that $X^+_U \cong X_U$ and that $K_{X_U} + B|_{X_U}$ is $\Q$-trivial over $U$.

Replacing $X$ with $X^+$, we can assume $h$ is equi-dimensional. Therefore, we can take a very ample divisor $S$ of $X$, which does not contain any component of $h^{-1}(T)$. We have the following commutative diagram
$$ \xymatrix@C=2cm{&S_1^{\nu} \ar[r]^>>>>>>>>>{\mathrm{normalization}}\ar[d]^{\phi_{S^{\nu}}}
&S_1 = \phi^{-1}S \ar[r]\ar[d]^{\phi_{S}} &X_1 \ar[r]^{h_1}\ar[d]^{\phi} & Z_1\ar[d]^{\psi}\\
&S^{\nu} \ar[r]^{\mathrm{normalization}}\ar[d]^{\sigma} &S \ar[r]&X\ar[r]^h & Z \\
&W\ar[rrru] &&&}.$$
where $W$ is introduced as follows.
The divisor $(K_X + B)|_{S^{\nu}}$ is nef and big. Consider the exceptional locus $E((K_X + B)|_{S^{\nu}})$, i.e., the union of finitely many $K_X + B$-numerically trivial curves on $S^{\nu}$. By the construction above, the image of $E((K_X + B)|_{S^{\nu}})$, via the natural map $S^{\nu} \to X$, is contained in finitely many fibers of $h$ over some closed points in $U$. So $(K_X + B)|_{E((K_X+B)|_{S^{\nu}})}$ is semi-ample, and by \cite[Theorem 1.9]{Ke99} $(K_X+B)|_{S^{\nu}}$ is semi-ample. For sufficiently divisible positive integer $n$, $|n(K_X+B)|_{S^{\nu}}|$ defines a birational morphism $S^{\nu} \to W$ to a normal projective variety $W$, and $(K_X+B)|_{S^{\nu}}$ descends to an ample divisor $(K_X+B)|_{W}$. And by the construction, the morphism $S^{\nu} \to Z$ factors through a finite morphism $W \to Z$.

First we will show that $Z$ is a projective variety. Since $W \setminus W_U$ consists of finitely many closed points, for $m \gg0$, we can take a Cartier divisor $D' \sim m(K_X+B)|_{W}$ contained in $W_U$. Let $D'_1 = \phi_{S^{\nu}}^* \sigma^*D$. Then $D_1'$ is supported in $S^{\nu}_U$ and $D_1'|_{E_{S_1^{\nu}/W}} \sim 0$ where $E_{S_1^{\nu}/W}$ denotes the exceptional locus of $S_1^{\nu} \to W$. Let $S_1^{\nu} \xrightarrow{\gamma} \bar{S}_1 \to Z_1$ be the Stein factorization. Obviously, $E_{S_1^{\nu}/\bar{S}_1} \subset E_{S_1^{\nu}/W}$, so $D_1'$ descends to a divisor $\bar{D}_1'$ on $\bar{S}_1$.
Taking the norm of $\bar{D}_1'$ over $Z_1$, we get a Cartier divisor $A'_1$ on $Z_1$ supported in $U$ (\cite[p.272 Remark 2.19 and p.274 ex. 2.6]{Liu10}). By the construction above, $A'_1$ is nef and big, $E(A'_1) = E(D_1)$ and $A_1'|_{E(A'_1)} \sim 0$. Thus $A_1'$ is a semi-ample divisor by \cite[Theorem 1.9]{Ke99}, to which the associated morphism coincides with $\psi: Z_1 \to Z$. Therefore, $Z$ is a projective variety.

Take $D \in |n(K_X+B)|$ for some sufficiently divisible integer $n>0$. Then there exists a Cartier divisor $A_U$ on $U$ such that $D_U = h^*A_U$. Let $A$ be the closure of $A_U$ in $Z$. Then $D= h^*A$. Taking the norm of $D|_W$ over $Z$, we get a Cartier divisor $dA$ where $d$ is the degree of the map $W \to Z$. Therefore, $A$ is a $\Q$-Cartier divisor. Finally by Nakai-Moishezon criterion, $A$ is an ample divisor on $Z$, which means that $K_X + B$ is semi-ample.
\end{proof}

\section{Separability of Albanese morphisms and Kodaira dimensions}\label{sec-ch-of-ab}
In this section, we first prove the following theorem, which is very useful to treat inseparable Albanese maps.
\begin{thm}\label{reduction}
Let $X$ be a smooth projective non-uniruled variety of dimension $n$. Denote by $a_X: X \to A_X$ the Albanese map. Assume that

(i) smooth resolution of singularities exists in dimension $n$;

(ii) $\dim a_X(X) \geq n-1$.

Then $\kappa(X) \geq 0$, and if the equality is attained then $a_X: X \to A_X$ is a separable surjective morphism.
\end{thm}

\medskip

Before the proof, we recall an easy result as a preparation.
\begin{lem}\label{rk-to-h0}
Let $X$ be a smooth projective variety and $V$ a torsion free coherent sheaf on $X$. Assume that $V$ is generically globally generated and $h^0(X, V) >  \rank ~V$. Then
$$h^0(X, (\det V)^{**}) \geq 2$$
where $(\det V)^{**}$ denotes the double dual of $\det V$.
\end{lem}
\begin{proof}
To compute $h^0(X, (\det V)^{**})$, we can work in codimension one. So we may shrink $X$ to assume that $V$ is locally free.

The case $\rank~ V = 1$  is trivial. We do induction and assume that the assertion is true for vector bundles satisfying our conditions and of rank smaller than $\rank~ V$. Take a basis $s_1, s_2, \cdots, s_k$ of $H^0(X, V)$. Let $W$ be the saturation of the sub-sheaf generated by $s_1, s_2, \cdots, s_{r-1}$. Then both $W$ and $V/W$ are generically globally generated, which can be assumed locally free by shrinking $X$ again. Then
$$h^0(X, \det W) \geq 1~\mathrm{and}~h^0(X, \det V/W) \geq 1,$$
and by the induction of ranks, at least one of the strict inequalities above is attained since
$$h^0(X, W) + h^0(X, V/W) \geq h^0(X, V) >  \rank ~V = \rank~ W + \rank ~V/W.$$
So the assertion follows easily by the relation $\det V \sim \det W \otimes \det V/W$.
\end{proof}

\begin{proof}[of Theorem \ref{reduction}]
We fall into two cases $\dim a_X(X) = n$ and $\dim a_X(X) = n-1$.

\medskip

\textbf{Case $\dim a_X(X) = n$.} In this case, if $a_X$ is separable then $\Omega_X^1$ is generically globally generated, thus $\kappa(X) \geq 0$. If moreover $\kappa(X) = 0$, then $a_X$ is surjective because otherwise, we will have $h^0(X, \Omega_X^1) > n = \rank ~\Omega_X^1$, which implies $h^0(X, \Omega_X^n) \geq 2$ by Lemma \ref{rk-to-h0}.

Let's consider the case that $a_X$ is inseparable. We will use the argument of \cite[ Prop. 4.3]{Ek87}, and do induction on $\deg a_X$. Assume that a variety $Y$ with Albanese map of degree $\leq \deg a_X -1$ has Kodaira dimension
$\kappa (Y) \geq 0$, and if the Albanese map is inseparable then $\kappa(Y) >0$.

Let $\mathcal{L}$ denote the saturation of the image of the natural homomorphism $a_X^*\Omega_{A_X}^1 \rightarrow \Omega_{X}^1$. By Igusa's result \cite[Theorem 4]{Se58}, $h^0(X, \mathcal{L})\geq h^0(A_X, \Omega_{A_X}^1) \geq n$. Since $\mathcal{L}$ is generically globally generated and $\mathrm{rank}~\mathcal{L} \leq n-1$ (Proposition \ref{prop-of-sep}), by Lemma \ref{rk-to-h0} we have
$$h^0(X, (\det \mathcal{L})^{**}) \geq 2.$$

We get a natural foliation $\mathcal{F} = \mathcal{L}^{\bot} \subset T_X$. Denote $Y = X/\mathcal{F}$. Then there is a natural morphism $a_Y: Y \to A_X$ fitting into the following commutative diagram
$$\centerline{\xymatrix{
&X \ar[d]_{a_X}\ar[r]^{\pi} &Y \ar[ld]^{a_Y}  \\
&A_X &
}}$$
Denote by $Y_0$ the smooth part of $Y$, and $X_0 = \pi^{-1}Y_0$. Then $\mathrm{codim}_X (X \setminus X_0) \geq 2$, $\mathcal{F}|_{X_0}$ is a smooth foliation on $X_0$, and by proposition \ref{flt}
$$(*) \hspace{1cm} K_{X_0} \sim \pi^*K_{Y_0} + (p-1)\det \mathcal{F}|_{X_0}.$$
On the other hand, we have the following exact sequence
$$0 \rightarrow  \mathcal{L}|_{X_0} \rightarrow \Omega_{X_0}^1 \rightarrow \mathcal{F}^*|_{X_0} \rightarrow 0,$$
which gives
$$\det\mathcal{F}|_{X_0} \sim \det \mathcal{L}|_{X_0} - K_{X_0}.$$
Combining $(*)$, we get
$$K_{X_0} \sim_{\mathbb{Q}} \frac{1}{p}(\pi^*K_{Y_0} + (p-1)\det \mathcal{L}|_{X_0}).$$
By $h^0(X_0, \det \mathcal{L}|_{X_0}) \geq 2$ and the induction that $\kappa(Y_0) \geq 0$, we show that $\kappa(X) >0$.

\medskip

\textbf{Case $\dim a_X(X) = n - 1$.} Let $a_X = a_Z \circ f: X \to Z \to A=A_X$ be the Stein factorization of $a_X$. Then $\kappa(Z) \geq 0$ by the previous case.

If $a_X$ is a separable morphism, then so is $f$. Since $X$ is non-uniruled, applying Corollary \ref{Iit-cor} to $f: X \to Z$, we conclude that $\kappa(X) \geq 0$, and the equality is attained only when $\kappa(Z) = 0$ and thus $a_X$ is surjective by the result of the previous case.

We are left to consider the case that $a_X$ is inseparable. Let $\mathcal{L}$ denote the saturation of the image of the natural homomorphism $a_X^*\Omega_{A_X}^1 \rightarrow \Omega_{X}^1$.
Then $\mathcal{L}$ is generically globally generated, $\mathrm{rank}~\mathcal{L} \leq n-2$ by Proposition \ref{prop-of-sep}, and $h^0(X, \mathcal{L})\geq h^0(A_X, \Omega_{A_X}^1) \geq n-1$ by Igusa's result again, which implies $h^0(X, (\det \mathcal{L})^{**}) \geq 2$ by Lemma \ref{rk-to-h0}.
We get a natural foliation $\mathcal{F} = \mathcal{L}^{\bot} \subset T_X$ of rank $\geq 2$, and a quotient map $\rho: X \rightarrow X_1 = X/\mathcal{F}$, which is a factor of $a_X$.
If $W_1$ is a smooth resolution of $X_1$, then as in the previous case we have that
\begin{itemize}
\item[$(\clubsuit)$]
$\kappa(X) \geq \kappa(W_1)$, and if $\kappa(W_1) = 0$ then $\kappa(X) \geq 1$.
\end{itemize}

Let $X'$ be the normalization of the reduction of $X^{(1)} \times_{Z^{(1)}} Z$. Then $X'$ is between $X$ and $X^{(1)}$, i.e., $F_X$ has the factorization
\begin{align*}
F_X: X \xrightarrow{\phi} X' \xrightarrow{\pi'} X^{(1)}.
\end{align*}
We claim that the natural morphism $\pi': X' \to X^{(1)}$ factors through a morphism $\pi_1: X'\to X_1$. Indeed, by Proposition \ref{flt} we can assume $X'= X/\mathcal{F}'$ for some foliation $\mathcal{F}'$ on $X'$, and we only need to show $\mathcal{F}'\subset \mathcal{F}$, which is equivalent to that $\mathcal{F}^{\bot} \subset \mathcal{F}'^{\bot}$. Note that $\mathcal{F}'^{\bot}$ and $\mathcal{F}^{\bot}$ coincide with the
saturation of the image of the natural maps $\phi^*\Omega_{X'} \to \Omega_X$ and $a_X^*\Omega_{A_X}^1 \rightarrow \Omega_{X}^1$ respectively. Then the claim follows from the fact that
$a_X$ factors through $\phi: X \to X'$.

We fit the above varieties into the following commutative diagram
$$\centerline{\xymatrix{
&X \ar[d]_{\phi}\ar[rd]_{\rho}\ar[rrd]^{F_X} & &\\
&X'\ar[r]_{\pi_1}\ar[d]_{f'} &X_1 \ar[r]_{\pi_2} \ar[d]_{f_1}  &X^{(1)}\ar[d]_{f^{(1)}} \\
&Z \ar[r]_{\pi_3}\ar[d]_{a_Z}  &Z_1 \ar[ld]^{a_{Z_1}}\ar[r]_{\pi_4}   &Z^{(1)}\ar[d]^{a_{Z^{(1)}}} \\
&A\ar[rr]^{F_A} &  &A^{(1)}
}}$$
where $f_1: X_1 \to Z_1$ is the fibration arising from the Stein factorization of the natural morphism $X_1 \to A$, and $\pi_i, \rho, f'$ denote the natural maps. Note that $f'$ is a fibraion by
$$\mathcal{O}_Z \subseteq f'_*\O_{X'} \subseteq f'_*\phi_*\O_{X} = f_*\O_{X} = \O_Z.$$

We fall into two cases.

Case I: $\deg \pi_3 > 1$. Then $\deg a_{Z_1} < \deg a_{Z}$.

Case II: $\deg \pi_3 = 1$, i.e., $Z_1 = Z$.

We claim that in case II,
$$(\spadesuit)~~~\mathrm{mult} (X'_{\ol\eta}) < \mathrm{mult} (X_{\ol\eta}).$$
Indeed, in Case II, by the universal property of fiber product, there is a natural dominant morphism $X_1 \to X^{(1)} \times_{Z^{(1)}} Z$, which implies that $\pi_2$ factors through a morphism $X_1 \to X'$. So $\pi_1: X' \to X_1$ is an isomorphism.
By $\rank~\mathcal{F} \geq 2$, we have $\deg \pi_2 =\frac{p^n}{\deg \rho}\leq p^{n-2} < \deg \pi_4$, which implies that $X^{(1)} \times_{Z^{(1)}} Z$ is not reduced.
Finally, comparing multiplicities of the geometric generic fibers of $f': X' \to Z$ and $f: X \to Z$, we can show the inequality $(\spadesuit)$.

We can run a program, beginning with the fibration $f_0 = f: W_0 =X \to Z_0 = Z$. Assume we have defined $W_{n-1}, Z_{n-1}$ and the fibration $f_{n-1}: W_{n-1} \to Z_{n-1}$. If $a_{W_{n-1}}: W_{n-1} \to A$ is inseparable, we can go the process above,
\begin{itemize}
\item
if in Case I, let $W_n$ be a smooth resolution of $(X_{n-1})_1$, and let $Z_n = (Z_{n-1})_1$;
\item
if in Case II, let $W_n$ be a smooth resolution of $(X_{n-1})'$, and let $Z_n = Z_{n-1}$.
\end{itemize}
We will end this program when $f_n: W_n \to Z_n$ is separable, equivalently the geometric generic fiber of $f_n$ is reduced by Proposition \ref{prop-of-sep}.

This program will terminate. Indeed, when running this program, we will fall into Case I for finitely many times since $\deg a_{Z_i}$ will be stable. So after finitely many steps, we always fall into Case II. But then since $\mathrm{mult}(W_i)_{\ol\eta_i} < \mathrm{mult}(W_{i -1})_{\ol\eta_{i-1}}$, after finitely many steps, we obtain a fibration $f_n: W_n \to Z_n$ having reduced geometric generic fiber, and the program terminates.

Applying Corollary \ref{Iit-cor}, we have $\kappa(W_n) \geq 0$.
Finally applying $(\clubsuit)$, by induction we can show that $\kappa(X) > 0$.
\end{proof}

As an easy application, we can characterize abelian varieties birationally by the conditions below. Note that it is expected that a smooth projective variety, with zero Kodaira dimension and maximal Albanese dimension, is birational to an abelian variety, which is finally proved in \cite{HPZ17} by much more technical arguments later.
\begin{thm}\label{ch-of-ab-kx}
Let $X$ be a smooth projective variety of maximal Albanese dimension of dimension $n$. Assume smooth resolution of singularities exists in dimension $\leq n$.
Then $\kappa(X) \geq 0$.

Moreover if $\kappa(X) = 0$ and either

(a) the Albanese map $a_X: X \to A_X$ factors into $a_X = a_{X'} \circ \sigma: X\dashrightarrow X' \to A_X$ where $\sigma: X \dashrightarrow X'$ is a birational map to  a minimal model $X'$ of $X$ with at most klt singularities, and $K_{X'} \sim_{\Q} 0$, or

(b) $\mathrm{char}~k = p>5$, $\dim X = 3$ and $A_X$ is simple, i.e., $A_X$ contains no non-trivial abelian varieties (\cite[IV. 19]{Mum74}), \\
then $X$ is birationally equivalent to an abelian variety.
\end{thm}
\begin{proof}
The first assertion follows from Theorem \ref{reduction}.

For the remaining assertions, we assume that $\kappa(X) = 0$. Then applying Theorem \ref{reduction} shows that $a_{X}$ is a separable surjective morphism.

In case (a), since $K_{X'} \sim_{\Q} 0$, by ramification formula we have that $a_{X'}$ is \'{e}tale in codimension one. Let $a_{X'} = a_{X''} \circ \sigma': X'\to X'' \to A_X$ be the Stein factorization where $\sigma'$ is a birational morphism. By the purity of branch locus theorem \cite[X, Theorem 3.1]{SGA1}, we see that $a_{X''}: X'' \to A$ is \'{e}tale, hence $X''$ is an abelian variety by \cite[Sec. IV 8 Theorem]{Mum74}.

In case (b), applying Theorem \ref{rel-mmp} (1) we have a relatively minimal model $a_{X'}: X'\to A_X$ over $A_X$, and $X'$ is minimal by Theorem \ref{rel-mmp} (3.4) since $A_X$ contains no rational curves. Denote by $R \subset X'$ the divisor corresponding to the pull-back of a nonzero global section of $\omega_A$ via $a_{X'}$.
Then $R \sim_{\Q} K_{X'}$ is nef. If $\dim a_{X'}(R) < 2$, then similarly as in case (a) we can show the theorem by \cite[X, Theorem 3.1]{SGA1}. From now on, assume that $\dim a_{X'}(R) = 2$. Take a component $T$ of $R$ such that $\dim a_{X'}(T) = 2$, and write that $R = nT + T'$ where $T$ is not a component of $T'$. Denote by $T^{\nu}$ the normalization of $T$. Then by adjunction \cite[5.3]{Ke99}, there exists an effective divisor $\Delta$ on $T^{\nu}$ such that
$$K_{T^{\nu}} + \Delta \sim_{\Q} (K_X + \frac{R}{n})|_{T^{\nu}} \sim_{\Q} (1+ \frac{1}{n})K_{X'}|_{T^{\nu}}.$$
Let $\mu: S \to T^{\nu}$ be a smooth resolution. Then we can write that
$$K_S + B_1 = \mu^*(K_{T^{\nu}}+ \Delta) + E$$
where $E, B_1$ are effective divisors on $S$ having no common components, so $E$ is $\mu$-exceptional.

We claim that $S$ is of general type. Indeed, since $S$ is of maximal Albanese dimension, if $\kappa(S) =0$ then $S$ is  birational to an abelian surface (\cite[Sec. 10]{Ba01}); and if
$\kappa(S) =1$ then the Iitaka fibration of $S$ is an elliptic fibration (\cite[Theorem 9.9]{Ba01}). However, neither of the cases above happen since $A_X$ is simple.

By Theorem \ref{ct}, we conclude that $\kappa(T^{\nu}, K_{T^{\nu}} + \Delta)\geq 2$, i.e.,
$K_{T^{\nu}} + \Delta$ is nef and big. Then by the relation (cf. \cite[Sec. 1.2]{Deb01})
$$(1+ \frac{1}{n})^2 \cdot \frac{1}{n} K_{X'}^3 = ((1+ \frac{1}{n})K_{X'})^2 \cdot \frac{R}{n} \geq ((1+ \frac{1}{n})K_{X'})^2 \cdot T =   (K_{T^{\nu}} + \Delta)^2 >0$$
we see that $K_{X'}$ is big. However this contradicts $\kappa(X) = 0$.
\end{proof}

\section{Proof of the Theorem \ref{abundance}}\label{sec-pf-abd}
We will prove Theorem \ref{abundance} in this section.
Applying Theorem \ref{s-amp-k=2}, we only need to show that either $\kappa(X, K_X + B) \geq 1$ or $K_X + B\sim_{\Q} 0$. We argue case by case according to the Albanese dimension of $X$.

If $\dim a_X(X) = 2, 3$ then applying Theorem \ref{reduction}, otherwise applying Corollary \ref{Iit-cor} to the fibration $f: X \to Y$ arising from the Stein factorization of $a_X$, we can show that $\kappa(X) \geq 0$. Since $B\geq 0$, in the following we only need to consider the case $\kappa(X) = 0$.

\subsection{The case $\dim a_X(X) = 1$}\label{albdim=1}
In this case, since $X$ is not uniruled and $g(Y) \geq 1$, applying Corollary \ref{Iit-cor}, the condition $\kappa(X) = 0$ implies that $g(Y) = 1$ and a smooth model $\tilde{X}_{\bar{\eta}}$ of geometric generic fiber has $\kappa(\tilde{X}_{\bar{\eta}}) = 0$. From Proposition \ref{prop-of-sep} and Remark \ref{bc-smthing}, we know that $f$ is separable, and there exists some $e$-th absolute Frobenius iteration $F_Y^e: Y'= Y^e \to Y$ such that, for a smooth resolution $X'$ of $X\times_Y Y'$, the geometric generic fiber $X'_{\ol\eta}$ of $f': X'\to Y'$ is a smooth surface with $\kappa(X'_{\bar{\eta}}) = 0$. Let $W$ be a relative minimal terminal model of $X'$ over $Y'$, and assume the birational map $\sigma: X'\to W$ is a morphism.
We fit the varieties above into the following commutative diagram
$$\xymatrix@C=2cm{
&X' \ar[rrd]^{\pi'}\ar[d]^{\sigma} \ar[rdd]|{f'}\ar[rd] &  &\\
&W\ar[rd]|g &X \times_Y Y'\ar[r]^<<<<<<{}\ar[d]    &X\ar[d]^f \\
& &Y'=Y^e\ar[r]^{F_Y^e}   &Y
}$$
where $g,  \pi'$ denote the natural morphisms.

Applying Proposition \ref{compds}, since $K_Y \sim 0$ there exist effective $\pi'$-exceptional divisors $E', E''$ and effective divisors $B', B''$ on $X'$ such that
$$K_{X'} + B' \sim_{\mathbb{Q}} \pi'^*(K_X +B) + E'~\mathrm{and}~K_{X'} + B'' \sim_{\mathbb{Q}} \pi'^*K_X + E''.$$
Applying Theorem \ref{ct} we have
$$\kappa(X', K_{X'}) \leq \kappa(X', \pi'^*K_X + E'') = \kappa(X, K_X) = 0,$$
in turn by Corollary \ref{Iit-cor} we obtain that $\kappa(X') = \kappa(W) = 0$.
Applying Theorem \ref{rel-mmp} (3.3, 3.4), we conclude $K_W \sim_{\Q} 0$.
Since $W$ is terminal, we can assume $K_{X'} \sim E$ where $E$ is an effective $\sigma$-exceptional divisor and contains all $\sigma$-exceptional divisorial components.
Then by Theorem \ref{ct}, for every integer $n>0$
$$\kappa(X, K_X + B) = \kappa(X', K_{X'} + B') = \kappa(X', nE + B').$$
Let $H = \sigma_*B'$. If $n$ is sufficiently large, then
$$\sigma^*H \leq  nE + B', ~\mathrm{thus}~\kappa(X, K_X + B) \geq \kappa(W, H).$$

Fix a rational number $t >0$ such that $(W, tH)$ is klt. Running an LMMP for $K_W + tH$ over $Y'$, after finitely many flips and divisorial contractions
$$(W, tH) = (W_0, tH_0) \dashrightarrow (W_1, tH_1) \dashrightarrow \cdots \dashrightarrow (W_n, tH_n) = (W', tH'),$$
we get a relative minimal model $(W', tH')$ and a fibration $g': W' \to Y'$.
By induction, for every step above the extremal ray is $K_{W_i}$ -trivial, thus $K_{W'} \sim_{\Q} 0$, and the geometric generic fiber of $g'$ has canonical singularities.
Applying Theorem \ref{rel-mmp} (3.4) we conclude that $H' \sim_{\Q} K_{W'}+ H'$ is nef.

By Corollary \ref{div-k=0}, we have that $\kappa(W', H') \geq 1$ unless $H' = 0$. If $H' > 0$, then we are done by
$$\kappa(X, K_X + B) \geq \kappa(W, H) = \kappa(W', H' )\geq 1.$$
If $H' = 0$, consider the map $\sigma': X'\to W'$ which may be assumed to be a morphism by blowing up $X'$. Then $\sigma'_*B'= 0$, and thus $K_{X'} + B' \sim E + B'$ is $\sigma'$-exceptional.
Take an effective $\Q$-divisor $D'\sim_{\Q} \pi'^*(K_X +B)$. Then $\mathrm{Supp}(D') \subset \mathrm{Supp}(E+B') $, thus $D'$ is $\sigma'$-exceptional. Combining the nefness of $D'$, we get that $D'=0$. Therefore $K_X +B\sim_{\Q} 0$, which completes the proof in this case.

\subsection{The case $\dim a_X(X) = 2$.}\label{albdim=2}
In this case, by Theorem \ref{reduction} we know that the fibration $f: X\to Y$ is separable.
By Corollary \ref{Iit-cor}, $\kappa(Y) = 0$, and the smooth model of geometric generic fiber $X_{\ol\eta}$ is a smooth elliptic curve. We see that $Y$ is birational to an abelian surface $A$, and since $Y \to A_X$ is finite, in turn, by the universality of Albanese maps we conclude that $Y = A = A_X$. For some Frobenius iteration $F_A: A^e \to A$ and a smooth resolution $X' \to X\times_A A^e$, the fibration $f': X' \to A^e$ has smooth geometric generic fiber. The proof of Corollary \ref{Iit-cor} also shows that $\kappa(X') = 0$, which combining Theorem \ref{Iit-conj} (iv) gives that $\mathrm{var}(f') = 0$. Therefore, there exists a generically finite surjective morphism $\pi: Z \to A$ with $Z$ smooth and projective, such that $X\times_{Y} Z$ is birational to $Z\times C$.

In the following by running an LMMP we can assume $K_X + B$ is nef and there exists an effective divisor $D \sim_{\Q} K_X +B$. We fall into two cases below.

\textbf{Case $\kappa(X_{\eta}, (K_X + B)_{\eta}) = 0$}: In this case applying \cite[Proposition 2.1]{Ka85I} and flattening trick (see also \cite[Lemma 5.6]{BW14}), we get the following commutative diagram
$$\xymatrix{&X'\ar[r]^{\phi}\ar[d]_{f'} &X\ar[d]^f\\
&Y'\ar[r]^{\psi} &A}$$
where $X'$ and $Y'$ are both smooth and projective, $\phi, \psi$ are birational morphisms such that,
there exists a $\mathbb{Q}$-divisor $H'$ on $Y'$ satisfying that $\phi^*(K_X+B) \sim_{\mathbb{Q}} f'^*H'$. Note that $K_{Y'} \sim E$ for some effective $\psi$-exceptional divisor $E$ containing every $\psi$-exceptional component.

We can assume $H'$ is effective. If $H' = 0$ then we are done. If $H' > 0$, then
$\psi_*H' > 0$ since $H'$ is nef. Since $\phi: X' \to X$ is birational, there exist two effective $\phi$-exceptional divisors $E_1', E_2'$ such that $K_{X'} + E_2' \sim_{\mathbb{Q}} \phi^*K_X + E_1'$.
Applying Theorem \ref{ct}, we can conclude that
\begin{align*}
\kappa(X, K_X + B) & = \kappa(X', \phi^*(m+1)(K_X + B))  \hspace{1cm} \text{for every $m>0$}\\
                   & \geq \kappa(X', \phi^*(K_X + B) + \phi^*mK_{X}) \\
                   & = \kappa(X', \phi^*(K_X + B) + m(\phi^*K_{X} + E_1')) \\
                   & \geq \kappa(X', \phi^*(K_X + B) + mK_{X'})  \\
                   & = \kappa(X', f'^*H' + mf'^*K_{Y'} + mK_{X'/Y'}) \\
                   & \geq \kappa(Y', H' + mK_{Y'})  \hspace{1cm} \text{since $\kappa(X',K_{X'/Y'}) \geq 0$ by Theorem \ref{rel-can-ellfib}}\\
                   & = \kappa(Y', H' + mE).
\end{align*}
For sufficiently large $m$ we have $\kappa(Y', H' + mE) = \kappa(A, \psi_*H') > 0$.

\textbf{Case $\kappa(X_{\eta}, (K_X + B)_{\eta}) = 1$}: In this case $D$ has a horizontal component $S$ over $Y$.
Consider the following commutative diagram
$$ \xymatrix@C=1.5cm{&W\ar[rrd]^{\rho}\ar[rd] \ar[d]^{\sigma}\ar@/_3pc/[dd]_{g} &  & \\
&W_1=Z\times C \ar[d]^{h} &W_2= X\times_{Y} Z \ar[r]\ar[d] &X\ar[d]^{f} \\
&Z\ar[r]^{{\mathrm{id}_Z}} & Z \ar[r]^{\pi} &Y=A }$$
where $W$ is a common smooth resolution of $W_1$ and $W_2$. Note that $K_{W/Z} \sim E_1$ for some effective $\sigma$-exceptional divisor $E_1$. And applying Proposition \ref{compds} there exists an effective divisor $E_2$ on $W$ exceptional over $X$, such that
$$E_1 \sim K_{W/Z} \leq \rho^*K_{X/A} + E_2 = \rho^*K_{X} + E_2.$$
Up to a further base change, we can assume $\sigma_*\rho^*S$ contains a section $T$ of $h$.
We conclude
\begin{align*}
\kappa(X, K_X + B) & = \kappa(X, 2K_X + B)   \\
                   & = \kappa(W, \rho^*(K_X + B) + \rho^*K_{X}+ E_2) \hspace{1cm} \text{by Theorem \ref{ct}} \\
                   & \geq \kappa(W, \rho^*(K_X + B) + E_1)  \\
                   & \geq \kappa(W, \rho^*D + E_1) = \kappa(W, \rho^*D + mE_1) \hspace{1cm} \text{for every $m>0$}\\
                   & = \kappa(W_1, \sigma_*\rho^*D) \hspace{1cm} \text{since $E_1$ contains every $\sigma$-exceptional component} \\
                   & \geq \kappa(W_1, T).
\end{align*}
Then since $T$ is a section $t: Z \to W_1$, we can get a projection $h': W_1 \to \mathrm{Pic}^0(C) \cong C$ by $(z, c) \mapsto c - p_2(t(z))$ where $p_2: W_1 = Z \times C \to C$ denotes the projection to the second factor. This gives another trivialization and $T$ is a fiber of $h'$. It follows easily that $\kappa(W_1, T) >0$.

\subsection{The case $\dim a_X(X) = 3$}\label{albdim=3}
Applying Theorem \ref{reduction} again, we can assume $a_X: X \rightarrow A_X$ is a separable surjective morphism.

If $A_X$ is simple, then $a_X: X \to A_X$ is a birational morphism by Theorem \ref{ch-of-ab-kx}. Considering the divisor $H = a_{X*} B$ and arguing as in the last paragraph of Sec. \ref{albdim=1}, we complete the proof.

If $A_X$ is not simple, then there exists a fibration $q: A_X \to A'$ where $A'$ is an elliptic curve or an abelian surface. We can assume there is no abelian variety between $A_X$ and $A'$, thus $q$ is separable (cf. \cite[Theorem 4]{Se58}). Considering the composite morphism $X\to A'$, the theorem can be proved by the same argument as in Sec. \ref{albdim=1} and \ref{albdim=2}.

\end{document}